\documentclass[reqno]{amsart}[12pt]
\usepackage{fullpage}
\usepackage{amsfonts,amssymb,amsthm,amsbsy,latexsym,amscd,amsmath,euscript,enumitem,manfnt, graphicx, mathtools, marvosym,verbatim,calc,mathrsfs,tensor,textcomp,color,xcolor,tikz-cd,setspace,upgreek,bbm,bm}
\usepackage{amsgen,graphicx,dsfont,textcomp,pifont}
\usepackage[citecolor=blue,colorlinks=true]{hyperref}
\usepackage[all]{xy}
\newtheorem{defn}{Definition}[section]
\newtheorem{prop}[defn]{Proposition}

\newtheorem{lem}[defn]{Lemma}
\newtheorem{thm}[defn]{Theorem}
\newtheorem{cor}[defn]{Corollary}

\newtheorem {conj}[defn]{Conjecture}

\newtheorem{claim}[defn]{Claim}

\newcommand {\C}{{\mathds C}}

\newcommand {\Q}{{\mathds Q}}

\def\mod{\operatorname{mod}}
\def\dim{\operatorname{dim}}

\title{ A base change version of Rasmussen-Tamagawa Conjecture}

\begin{document}
	
	\author{Plawan Das}

	\address{Chennai Mathematical Institute, H1 Sipcot IT Park, Siruseri, Kelambakkam 603 103, India}  
	\email{d.plawan@gmail.com, plawan@cmi.ac.in}
	
	\author{Subham Sarkar}
	
	\address{School of Mathematical Sciences, National Institute of Science Education and Research, Bhubaneswar, Khurda 752050, India.}  \email{subham.sarkar13@gmail.com, subham13@niser.ac.in}

	\subjclass{Primary 11F80; Secondary 11G05, 11G10}
	
	\begin{abstract} 
		We prove a certain uniform version of the Shafarevich Conjecture. As a corollary, we prove the Rasmussen-Tamagawa Conjecture for a particular class of abelian varieties $A$  defined over a number $K$ of dimension $g$ having everywhere potential good reduction, in particular, for any finite place $v$ of $K$ the localization $A_v:=A\times_{\mathrm{Spec}(K)}\mathrm{Spec}(K_v)$ has either good reduction or {\it totally bad reduction} (connected component $\tilde{\mathcal{A}}_v^0$ of  the special fibre $\tilde{\mathcal{A}}_v$ of the N\'eron model $\mathcal{A}_v$ at $v$ is an affine group scheme over the residue field $k_v$ at $v$) and has good reduction over a quadratic extension of $K_v$.

	\end{abstract}
	
	\baselineskip=12pt
	\maketitle

	\section{Introduction}
	
	Let $K$ be a number field, $S$ be a finite set of  places of $K$ containing the archimedean places and $g$ be a positive integer. Let $\bar{K}$ be an algebraic closure of $K$ and $G_K$ denote the absolute Galois group of $K$.  For a finite place $v$ of $K$, let $K_v$ denote the completion  of $K$ at $v$ and $k_v$ denote the residue field at $v$. Let $A/K$ denote an abelian variety defined over $K$ and  $A_v=A\times_{\mathrm{Spec}(K)}\mathrm{Spec}(K_v)$ denote the localization of $A$ at $v$. Let $\mathcal{A}_v$ denote the N\'eron model of $A$ at $v$. Let  $\tilde{\mathcal{A}}_v$ be the special fibre of  the N\'eron model $\mathcal{A}_v$ and  $\tilde{\mathcal{A}}_v^0$ be the connected component of identity of $\tilde{\mathcal{A}}_v$. Recall that an abelian variety $A$ defined over $K$ is said to have {\it totally bad reduction} at $v$ if $\tilde{\mathcal{A}}_v^0$ is an affine group scheme over $k_v$ (see \cite{DR1}) and $A$ is said to have {\em purely additive reduction} if $\tilde{\mathcal{A}}_v^0$ is a unipotent group scheme. Note that when $A$ is an elliptic curve totally bad reduction is the same as bad reduction and purely additive reduction is the same as additive reduction.
	
	By Shafarevich Conjecture, proved by Faltings \cite{Fa83}, the $K$-isomorphism classes of abelian varieties $A/K$  of dimension $g$ having good reduction outside $S$ is finite. In \cite{DR23}, a uniform version of the Shafarevich conjecture is studied, namely, the finiteness of $\bar{K}$-isogeny classes of such abelian varieties allowing a certain type of bad reduction outside $S$. For any finite set of places $S'$ of $K$ containing $S$, let $\mathcal{X}(K,g,S')$ and $\bar{\mathcal{X}}(K,g,S')$  denote the $K$-isomorphism classes and $\bar{K}$-isomorphism classes of the abelian varieties $A/K$ of dimension $g$ respectively,  satisfying the following: 
	\begin{itemize}
		\item[(a)] for any place $v\notin S'$, $A$ has good reduction at $v$,
		\item[(b)] for any finite place $v\in S'\setminus S$, $A_v$ has either good reduction or totally bad reduction and acquires good reduction over a quadratic extension of $K_v$.
	\end{itemize}
	For an abelian variety $A/K$ of dimension $g$, let $\rho_{A,\ell}: G_K\rightarrow GL(V_\ell(A))$ denote the $\ell$-adic  representation associated to the $\ell$-adic Tate module $V_\ell(A)$ of $A$ which is of dimension $2g$. Note that, if $A_v$ has totally bad reduction and acquires good reduction over a quadratic extension of $K_v$, then  for any rational prime $\ell$ coprime to the residue characteristic of $K_v$, the inertia group $I_v$ at $v$ acts via scalars $\{\pm1\}$ on the Tate module $V_\ell(A_v)$ (see \cite[Corollary 7]{DR1}, Proposition \ref{DR1_cor7}).

	By Faltings' finiteness theorem \cite{Fa83}, we know that the set $\mathcal{X}(K,g,S')$ is finite. Hence  $\bar{\mathcal{X}}(K,g,S')$ is also finite. Further, for any abelian variety $A$ defined over a number field $K$, there are only a finite number of $\bar{K}$-isomorphism classes of abelian varieties $B$ defined over $K$ and $\bar{K}$-isogenous to $A$ (see Bost \cite[Corollary 3.2]{Bos96}). Hence \cite{DR23} gives that  for fixed $K,g$ and $S$, the union
	$$\bar{\mathcal{Y}}(K,g,S):=\bigcup_{S'\supseteq S}\bar{\mathcal{X}}(K,g,S')$$
	is finite (see Proposition \ref{fin_pot__isom}). Note that all abelian varieties $A$ satisfying conditions $(a)$ and $(b)$ have potential good reduction outside $S$, but there may not exist a uniform extension $L$ of $K$ such that $A$ have good reduction outside  places of $L$ lying above $S$. In fact, in \cite{DR23}, the authors proved that, up to potential equivalence, the collection of pure semisimple continuous $\ell$-adic representations  $\rho:G_K\rightarrow GL_n(\mathbb{Q}_\ell)$ such that the inertia at places outside $S$ acts by a scalar matrix is finite. Then using the isogeny theorem of Tate-Zarhin-Faltings (\cite{Fa83}): $\text{Hom}(A,B)\otimes \Q_{\ell}\simeq  \text{Hom}_{G_K}(V_{\ell}(A), V_{\ell}(B))$, for abelian varieties $A, B$ over $K$ of dimension $g$, they achieve the finiteness of $\bar{K}$-isogeny classes of abelian varieties.
	
	For any rational prime $\ell$, if the only finite places in $S$ are those lying above $\ell$, then we denote $\bar{\mathcal{Y}}(K,g,S)$ by $\bar{\mathcal{Y}}(K,g,\ell)$. Here we are interested in the finiteness of the union of a particular subcollection $\bar{\mathscr{M}}(K,g,\ell)$ of $\bar{\mathcal{Y}}(K,g,\ell)$ varying $\ell$ over all rational primes which is motivated by a conjecture of Rasmussen and Tamagawa \cite{RT08}. The conjecture is widely open and here we prove a special case of it.

	For any positive integer $n$, let $A[\ell^n]$ denote the group of $\ell^n$-torsion elements of $A$ defined over $\bar{K}$ and $K(\mu_{\ell})$ be  the $\ell$-th cyclotomic extension of $K$.
	Motivated by a question of Ihara \cite{Ih86}, Rasmussen and Tamagawa considered the set of $K$-isomorphism classes $[A]$ of abelian varieties $A/K$,
	\begin{align*}
		\mathcal{A}(K,g,\ell):=\{[A]: \dim A=g,~ A~\text{has good reduction outside}~\ell~& \text{and}\\~ K(A[\ell])/K(\mu_\ell)~ \text{is an}~ \ell\text{-extension}  \}.
	\end{align*}
	Given $K, g, \ell$, by Faltings' finiteness theorem, we have that $\mathcal{A}(K,g,\ell)$ is finite. They conjectured the following: 
	
	\begin{conj}[Rasmussen-Tamagawa Conjecture]\cite{RT08,RT16}\label{co1}
		For a number field $K$ and a positive integer $g$, the set  $\mathcal{A}(K,g,\ell)$ is empty for sufficiently large $\ell$. In particular, $$\mathcal{A}(K,g):= \bigcup_{\ell} \mathcal{A}(K,g,\ell)$$ is finite.
	\end{conj}
	The conjecture above can also be viewed as a uniform version of the Shafarevich Conjecture. Rasmussen-Tamagawa gave a proof of the conjecture for $K=\mathbb{Q}$ and $g=1$ (\cite[Theorem 2]{RT08}), $g=2, 3$ (\cite[Theorem 7.1 and Theorem 7.2]{RT16}). They further stated a uniform version of the Conjecture (\cite[Conjecture 2]{RT16}), in particular they conjectured that for any $g$ and any $n$, there exists a constant $N=N(g,n)>0$ such that $ \mathcal{A}(K,g,\ell)$ is empty for any $K$ with $[K:\mathbb{Q}]=n$ for   $\ell>N$ and assuming the Generalized Riemann Hypothesis they proved it when $n$ is odd. For any $n$, Bourdon \cite[Corollary 1]{Bo15} and Lombardo \cite[Theorem 1.4]{Lo18} proved the uniform version of the conjecture for CM elliptic curves and abelian varieties of CM type (i.e. $\mathrm{End}_{\bar{K}}(A)\otimes\mathbb{Q}$ contains an \'etale $\mathbb{Q}$-algebra of dimension $2g$) respectively. Further Arai \cite{Ar14} explored the Conjecture for QM-abelian surfaces.  In \cite[Corollary 1.4]{Oz13}, Ozeki proved the conjecture for  abelian varieties $A/K$ when the image of $\rho_{A,\ell}$ is abelian. For fixed $K,g$, Rasmussen-Tamagawa (\cite[Theorem 3.6]{RT16}, also see Ozeki \cite[Corollary 4.5]{Oz11}) further proved that  $$\mathcal{A}^{\mathrm{ss}}(K,g):= \bigcup_\ell \mathcal{A}^{\mathrm{ss}}(K,g,\ell)$$ is finite, where $\mathcal{A}^{\mathrm{ss}}(K,g,\ell)$ denote the subset of $ \mathcal{A}(K,g,\ell)$ containing only the $K$-isomorphism classes of abelian varieties $A/K$ with everywhere semistable reduction over $K$. 
	
	We focus on a very special class of abelian varieties $A\in \mathcal{A}(K,g,\ell)$  which have potential semistable reduction at places lying above $\ell$. In particular, we considered a collection $\mathcal{Z}(K,g)$ of abelian varieties $A$ defined over a number field $K$ of dimension $g$ having everywhere potential good reduction as defined below which satisfies the hypothesis of the Rasmussen-Tamagawa conjecture after a base chance to $K(A[12])$; and prove a particular uniform version of the Shafarevich Conjecture which we call a base change version of the Rasmussen-Tamagawa conjecture. 
	
	\noindent	\begin{defn}
		For any number field $K$ and a positive integer $g$, let $\mathcal{Z}(K,g)$ be the set of $K$-isomorphism classes  of abelian varieties $A/K$ of dimension $g$ such that for any finite place $v$ of $K$,  $A_v$ has either good reduction, or totally bad reduction  and acquires good reduction over a quadratic extension of $K_v$. So abelian varieties in $\mathcal{Z}(K,g)$ have everywhere potential good reduction.
	\end{defn}

	\begin{defn}\label{def:M} For any number field $K$, a rational prime $\ell$, and a positive integer $g$,
		let $\mathscr{M}(K,g,\ell)$ be a subset of $\mathcal{Z}(K,g)$ consisting of $K$-isomorphism classes of those abelian varieties $A/K$ for which $K_A(A[\ell])$ is an $\ell$-extension of $K_A(\mu_\ell)$, where $K_A=K(A[12])$.	Let $\bar{\mathscr{M}}(K,g,\ell)$ denotes the $\bar{K}$-isomorphism classes of such abelian varieties.
	\end{defn}
	
		Note that if a smooth proper variety $X/K$ has semistable reduction at $u$ dividing $\ell$ then the $\ell$-adic cohomology group $H^r_{\text{\'{e}t}}(X_{\bar{K}},\mathbb{Q}_\ell)$ and its dual are `semistable' at $u$ (see \cite[\S 5.1.4]{Fo94}). This is a consequence of Fontaine-Janssen conjecture (\cite[Conj. 6.2.7]{Fo94}) proved by Tsuji \cite{Ts99}. For an abelian variety $A/K$, by Raynaud's criteria $A$ has semistable reduction everywhere over $K(A[12])$ (\cite[Proposition 4.7]{SGA}).   
	
		For a positive integer $m$, here we introduce the notion of $m$-compatible system for a collection of $\ell$-adic representations, generalize a structure theorem due to Ozeki \cite{Oz11} for compatible systems of $\ell$-adic representations to $m$-compatible systems; and prove the following:
	
	\begin{thm}\label{thm1}
		Let $K$ be a number field and $g$ be a positive integer. For sufficiently large prime numbers $\ell$,  the set $\mathscr{M}(K,g,\ell)=\emptyset$. 
	\end{thm}
\noindent As a corollary we get the following  uniform version of the Shafarevich Conjecture: 
	\begin{cor}\label{Main:cor1}
		For a given number field $K$ and a positive integer $g$, the set $\bar{\mathscr{M}}(K,g):=\bigcup_{\ell}\bar{\mathscr{M}}(K,g,\ell)$ is finite.
	\end{cor}

	Motivated by a remark of Rasmussen-Tamagawa \cite[page 2411]{RT16}, here we studied the $\bar{K}$-isomorphism classes of abelian varieties satisfying the hypothesis of the conjecture. Now consider the set $\mathcal{A}^{\text{pot}}(K,g,\ell):=\mathcal{A}(K,g,\ell) \cap \mathcal{Z}(K,g)$. In fact we will see that $\mathcal{A}^{\text{pot}}(K,g,\ell)\subset \mathscr{M}(K,g,\ell)$ and hence we have the following Corollary.

	\begin{cor}\label{Main:cor2}
		Let $K$ be a number field and $g$ be a positive integer.  For sufficiently large prime numbers $\ell$, the set	$	\mathcal{A}^{\text{pot}}(K,g,\ell)$ is empty. In particular, the set $$	\mathcal{A}^{\text{pot}}(K,g):=\bigcup_\ell	\mathcal{A}^{\text{pot}}(K,g,\ell)$$ is finite.
	\end{cor} 
 Hence the Rasmussen-Tamagawa Conjecture is true for abelian varieties $A/K$ such that $[A]\in \mathcal{Z}(K,g)$.

	\section{Notation}
	For any field $F$, let $\bar{F}$ denote an algebraic closure of $F$ and  $G_F:=\mathrm{Gal}(\bar{F}/F)$. Let $K$ be a number field.   
	For any finite place $v$ of $K$,  $G_v:=\mathrm{Gal}(\bar{K}_v/K_v)$, where $K_v$ denote the completion of $K$ at $v$. Let's denote by $\mathcal{O}_v$ the ring of integers of $K_v$ and  $k_v$ denote the residue field of $K_v$. Let $q_v$ denote the cardinality of $k_v$. Let $I_v$ denote the inertia subgroup of $G_v$ corresponding to an extension $\bar{v}$ of $v$ to $\bar{K}$. Given a continuous $\ell$-adic representation  $\rho: ~G_K\to GL_n(\mathbb{Q}_\ell)$ and a place	$v $ of $K$ where $\rho$ is unramified, let $\rho(F_v)$ denote	the Frobenius conjugacy class in the image group $  G_K/{\rm Ker}(\rho) \simeq \rho( G_K) \subset GL_n(\mathbb{Q}_\ell)$.  By abuse of	notation, we will also continue to denote by \( \rho (F_v) \) an	element in the associated conjugacy class.

	\section{$m$-compatible system}\label{m-compatible}
	Let $E$ be a number field and  $\mathcal{O}_E$ be its ring of integers. Suppose $\lambda$ be a finite place  of $E$ and $\ell_\lambda$ be the rational prime lying below $\lambda$. Let $E_\lambda$  be the completion of $E$ at $\lambda$ and $\mathbb{F}_\lambda$ be the residue field at $\lambda$. Let $m$ be a positive integer.

	Given any continuous $\lambda$-adic representation $\rho_\lambda:G_K\rightarrow GL_n(E_\lambda)$, we call a set $S$ of finite places of $K$ a \textbf{$m$-ramification set} if $\rho_{\lambda}(I_v)$ is a group of order dividing $m$ for $v\notin S_{\ell_\lambda}$,  where $ S_{\ell_\lambda}$ denote the union of the set $S$  with the set of  places of $K$  dividing $\ell_\lambda$. Suppose $w$ is a place of $\bar{K}$ lying above $v$ not in $ S_{\ell_\lambda}$. Let the inertia at $w$ acts by a finite order matrix of order dividing $m$. As we vary the place $w$ lying above $v$, the conjugacy classes of Frobenius elements $\rho_{\lambda}(F_w )$ are well-defined only up to finite order matrices of order dividing $m$ and the conjugacy class $\rho_{\lambda}(F_v^m)$ is well-defined. Let $T$ be a finite set of finite places of $E$.
	
	\begin{defn}
		
		\begin{itemize}
			An \textbf{$E$-rational ($E$-integral) strictly $m$-compatible system $(\rho_\lambda)_\lambda$ of $n$-dimensional $\lambda$-adic representations of $G_K$ with defect set $T$ and finite $m$-ramification set $S$} is a family of   continuous (semisimple) representations
			$\rho_\lambda: G_K\rightarrow GL_n(E_\lambda)$ such that for any finite place $\lambda$ of $E$ not in $T$, $\rho_\lambda$   satisfy the following conditions:

			\begin{itemize}
				\item[1.] the group $\rho_{\lambda}(I_v)$ has order dividing $m$ for finite places $v\not\in S_{\ell_\lambda}$;
				\item[2.] for a finite place $v$ of $K$ not in $S$, there exists a monic polynomial $f_v[X]\in E[X]$ \emph{(resp. $f_v(X)\in \mathcal{O}_E[X]$)} such that for all finite places $\lambda\notin T$ of $E$ coprime to the residue characteristic of $v$, the characteristic polynomial of $\rho_\lambda(F_v^m)$ is equal to $f_v(X)$. 
			\end{itemize}
			\item 
			An \textbf{ $E$-rational ($E$-integral) strictly $m$-compatible system $(\bar{\rho}_\lambda)_\lambda$ of $n$-dimensional mod $\lambda$ representations of $G_K$ with defect set $T$ and a finite $m$-ramification set $S$} is a family of   continuous (semisimple) representations
			$\bar{\rho}_\lambda: G_K\rightarrow GL_n(\mathbb{F}_\lambda)$ such that for any finite place $\lambda$ of $E$ not in $T$, $\bar{\rho}_\lambda$   satisfy the following conditions:

			\begin{itemize}
				\item[1.] the group $\bar{\rho}_{\lambda}(I_v)$ has order dividing $m$ for finite places $v\not\in S_{\ell_\lambda}$;
				\item[2.] for a finite place $v$ of $K$ not in $S$ there exists a monic polynomial $f_v[X]\in E[X]$ \emph{(resp. $f_v(X)\in \mathcal{O}_E[X]$)} such that for all finite places $\lambda\notin T$ of $E$ characteristic coprime to the residue characteristic of $v$,  $f_v(X)$ is integral at $\lambda$ and the characteristic polynomial of $\bar{\rho}_{\lambda}(F_v^m)$ is equal to  the reduction of  $f_v(X)$  $\mod{\lambda}$.
			\end{itemize} 
		\end{itemize}

	\end{defn}
	
	\subsection{Raynaud's criterion of semistable reduction:}
	\begin{prop}\cite[Proposition 4.7]{SGA}\label{Raynaud}
		Suppose $A$ be an abelian variety over a field $F$ with discrete valuation $v$, $n$ is a positive integer not divisible by the residue characteristic, and the points of $A[n]$ are defined over an extension of $F$ which is unramified over $v$. If $n\geq 3$ then $A$ has semistable reduction at $v$. 
		
	\end{prop}
	
	Let $A$ be an abelian variety over a number field $K$ of dimension $g$. For any positive integer $n$, let $K(A[n])$ denote the minimal field generated by $K$ and the coordinates of all the $n$-torsion points of $A$ over $\bar{K}$. So the field extension $K(A[n])/K$ is Galois (\cite{ST68}). We have the following Galois representation associated with $A$
	$$\bar{\rho}_{A,n}:G_K\rightarrow GL_{2g}(\mathbb{Z}/n\mathbb{Z}),$$
	which gives an action of $G_K$ on $A[n]$. So we have, $$\textrm{im}(\bar{\rho}_{A,n})\cong\textrm{Gal}(K(A[n])/K).$$
	Now as a consequence of Proposition \ref{Raynaud}, $A$ has semistable
	reduction over $K(A[3])$ outside $3$ (take $F=K(A[3])$) and has  semistable reduction over $K(A[4])$ outside $2$ (take $F=K(A[4])$). Hence $A$ has semistable reduction everywhere over $K(A[12])$.   For any $g>0$, put  
	$$D_g=|GL_{2g}(\mathbb{Z}/3\mathbb{Z})|\cdot |GL_{2g}(\mathbb{Z}/4\mathbb{Z})|.$$
	Since we have
	$\textrm{im}(\bar{\rho}_{A,12})\cong \textrm{Gal}(K(A[12])/K)$, the degree $[K(A[12]):K]~|~D_g$.

	\subsection{Examples:}\label{Dg}
	1.	Any compatible system of $\lambda$-adic representations can be seen as a $1$-compatible system.

	2.	For any abelian variety $A/K$, let $S$ be a set of places of $K$ such that outside $S$ the abelian variety $A$ has potential good reduction. Consider a prime number $\ell$ and
	take a finite place $v\notin S_\ell$. 
	Let $L=K(A[12])$. Suppose  $v_L$ be a place of $L$ lying above $v$ and $f_L=[k_{v_L}:k_v] $, where $k_{v_L}$ and $k_v$ are the residue fields at $v_L$ and $v$ respectively.  Let $A$ have potential good reduction at $v$, so by Raynaud's semistability criteria, the base change $A_L$ have good reduction at $v_L$ and we have  $\mathrm{ord}(\rho(I_v))|D_g$.  Since $L$ is Galois extension of $K$ and $[L:K]$ divides $D_g$,  $D_g/f_L$ is an integer. We get the following equation 
	\begin{equation}\label{eq:D_g}
		\textrm{det}(XI_{2g}-\rho_{A,\ell}(F_v^{D_g}))=\textrm{det}(XI_{2g}-\rho_{A,\ell}(F_{v_L}^{D_g/f_L})).
	\end{equation}
	Since $A_L$ have good reduction at $v_L$, the polynomial $\textrm{det}(XI_{2g}-\rho_{A,\ell}(F_{v_L}))\in \mathbb{Z}[X]$ (see \cite{ST68}), hence the characteristic polynomial of $\rho_{A,\ell}(F_v^{D_g})$ is defined over $\mathbb{Z}$. Therefore, $(\rho_{A,\ell})_\ell$ forms a $D_g$-compatible $\mathbb{Q}$-integral system of $\ell$-adic representations with $D_g$-ramification set $S$ and defect set $T=\emptyset$.

	3. Isogenous abelian varieties $A/K, B/K$ of dimension $g$ give rise to equivalent compatible systems $(\rho_{A,\ell})_\ell$ and $(\rho_{B,\ell})_\ell$, i.e., for any $\ell$, $\rho_{A,\ell}\cong \rho_{B,\ell}$ with the ramification set as the set of bad places of $A$ or $B$ which are in fact the same. Let us construct a $m$-compatible system attached to a family of abelian varieties $A/K$ of dimension $g$ which are $\bar{K}$-isogenous. Consider a $\bar{K}$-isogeny class $\Gamma$ of abelian varieties which may contain infinitely many abelian varieties $A/K$ of dimension $g$ up to $K$-isomorphism. Take a finite set $S$ such that outside $S$ the abelian varieties $A \in \Gamma$ have potential good reduction.  By \cite{ST68}, $S$ contains all the places $v$ of $K$ such that $I_v$ acts on $V_\ell(A)$ via quasi-unipotent matrices (via infinite order matrices) for any $\ell$ coprime to $\mathrm{char}(k_v)$. Note that for any two $\bar{K}$-isogenous abelian varieties $A/K$, $B/K$ the finite places $v$ where $I_v$ acts via quasi-unipotent matrices  on $V_\ell(A)$ or $V_\ell(B)$  for any $\ell$-coprime to $\mathrm{char}(k_v)$ are the same, hence collection of such places $v$ for a $\bar{K}$-isogeny class is finite.

	To any pair $(\ell, A)$, $A\in \Gamma$, attach the $\ell$-adic representation $\rho_{\ell}:=\rho_{A,\ell}$. Then $(\rho_{\ell})_\ell$ gives  a $D_g^2$-compatible system with $D_g^2$-ramification set $S$ and defect set $T=\emptyset$. To see this, take any two abelian varieties $A, B$ that are isogenous over $\bar{K}$. Then by \cite[Theorem 2.4]{Si92}, they are isogenous over $L=K(A[12], B[12])$.   Since the Galois extension $L/K$ is a compositum of the Galois extensions   $K(A[12])/K$ and  $K(B[12])/K$, we have 
	$$[L:K]~\mid~ [K(A[12]):K] \cdot [K(B[12]):K],$$ which divides $D_g^2$. Therefore  $D_g^2/f_L$ is an integer. Using Raynaud's criteria,  $A$ and $B$ have good reduction outside $S_L$, where $S_L$ denote the places of $L$ lying above $S$. Note that the abelian varieties $A_L$ and $B_L$  give rise to  compatible systems of $\ell$-adic representations $(\rho_{A_L,\ell})_\ell$ and $(\rho_{B_L,\ell})_\ell$ respectively with ramification set $S_L$. Take two prime numbers $\ell, \ell'$ and a finite place $v\notin S\cup \{\text{places lying above}~\ell, \ell'\}$. Since $A$ and $B$ are isogenous over $L$, the representations $\rho_{A,\ell}$ and $\rho_{B,\ell}$ are equivalent when restricted to the subgroup $G_L$ of $G_K$. Further using compatibility of the system $(\rho_{B_L,\ell})_\ell$ we get
	$$\textrm{det}(XI_{2g}-\rho_{A,\ell}(F_{v_L}))=\textrm{det}(XI_{2g}-\rho_{B,\ell}(F_{v_L}))=\textrm{det}(XI_{2g}-\rho_{B,\ell'}(F_{v_L}))\in \mathbb{Z}[X],$$
	where $v_L$ is a place of $L$ lying above $v$.
	
	Now for representations $\rho_\ell:=\rho_{A,\ell}$ and $\rho_{\ell'}:=\rho_{B,\ell'}$ attached to the pairs $(\ell, A)$, $(\ell',B)$ respectively, where $A, B\in \Gamma$, we have the following equations 
	\begin{align*}
		\textrm{det}(XI_{2g}-\rho_{\ell}(F_v^{D_g^2})) =\textrm{det}(XI_{2g}-\rho_{A,\ell}(F_{v_L}^{D_g^2/f_L}))=&
		\textrm{det}(XI_{2g}-\rho_{B,\ell'}(F_{v_L}^{D_g^2/f_L}))\\=&
		\textrm{det}(XI_{2g}-\rho_{\ell'}(F_v^{D_g^2})).
	\end{align*}
	Hence, $(\rho_{\ell})_{\ell}$ give a $D_g^2$-compatible system with $D_g^2$-ramification set $S$ and defect set $T=\emptyset$.

	\section{ Weights}
	Let $\ell$ be a rational prime.  We fix an algebraic closure $\bar{\mathbb{Q}}_\ell$ of $\mathbb{Q}_\ell$. Let $\mathbb{C}_\ell$ denote the completion of $\bar{\mathbb{Q}}_\ell$. 
	Consider an $\ell$-adic representation $(\rho, V):G_K\rightarrow GL_{n}(\mathbb{Q}_\ell)$ and $(\bar{\rho}, \bar{V}):G_K\rightarrow GL_n(\mathbb{F}_\ell)$ be its residual representation corresponding to a $\mathbb{Z}_\ell$-lattice $\mathbb{L}$. Let $\bar{\rho}^{ss}$ denote the semisimplification of $\bar{\rho}$ which is independent of $\mathbb{L}$. Here we recall various definitions of weights attached to $\rho$ and $\bar{\rho}^{ss}$. We further discuss how they behave with base change. Let $v,u$ be the finite places of $K$ such that $v\nmid \ell$ and $u\mid \ell$.  Let $\chi_\ell: G_\mathbb{Q}\rightarrow \mathbb{Z}_\ell^*$ denote the cyclotomic character.

	\subsection{Weil Weights.}
	
	Choose any lift $F_v\in G_v$ of the Frobenius element $\mathrm{Frob}_v\in G_{k_v}$. Let the cardinality of the residue field $k_v$ be  $q_v$. Let $w$ be an integer. We call an algebraic integer $\alpha \in \bar{\Q}$ a \textbf{$q_v$-Weil number  of weight $w$} if for any embedding {$\phi: \bar{\Q}\to \C$, we have $|\phi(\alpha)|=q_v^{w/2}$.  Suppose the roots of the characteristic polynomial $P_v(T)$ of $\rho(F_v)$ are $q_v$-Weil number of weights   $w_1,\cdots,w_n$. We call the integers $w_1,\cdots, w_n$ the \textbf{Weil weights} of $\rho$  at $v$. The representation  $\rho$   is said to be \textbf{pure} of weight $w$ and degree at most $d$ at $v$, if $w_1=\cdots=w_n=w$ and  the  degree of the roots of $P_v(T)$ is at most $d$ over $\mathbb{Q}$. Note that the Weil weights of $\rho$ at $v$ is invariant under a finite extension of the base field $K_v$. By Serre-Tate (\cite[Theorem 3]{ST68}), for any abelian variety $A/K$, a place $v$ where $A$ has potential good reduction, the representation $\rho_{A,\ell}$ is pure of weight $1$ at $v$ for any $\ell \neq \mathrm{char}(k_v)$.  In fact,  the characteristic polynomial $P_v(T)$ of $ \rho_{A,\ell}(F_v)$ is defined over $\mathbb{Z}$.

		\subsection{Hodge-Tate Weights.}
		Let $\sigma$ denotes the unique continuous action of $G_{\mathbb{Q}_\ell}$ on $\mathbb{C}_\ell$ extending the action on $\bar{\mathbb{Q}}_\ell$. Set
		$$V(i)=\{v\in V\otimes_{\mathbb{Q}_\ell} \mathbb{C}_\ell~|~(\rho\otimes \sigma)(g)v=\chi_\ell^i(g)v~\text{for all}~ g\in G_u\}.$$
		The representation $(\rho, V)$ is \textbf{Hodge-Tate} if the natural map $\bigoplus V(i)\otimes_{\mathbb{Q}_\ell} \mathbb{C}_\ell \rightarrow V\otimes_{\mathbb{Q}_\ell} \mathbb{C}_\ell$ is an isomorphism and an integer $i$ is said to be \textbf{Hodge-Tate weight} of $\rho$ at $u$ if $V(i)\neq 0$. Note that the Hodge-Tate weights of $\rho$ at $u$  are invariant under any finite extension of $K_u$. 
		
		\subsection{Tame Inertia Weights.}(\cite{Ser72}, \cite{OT14}) 
		The tame inertia group $I_u^t$ of $K$ at $u$ is defined to be the 
		quotient of the inertia group $I_u$ by its maximal pro-$\ell$ subgroup. Note that for a positive integer $h$, any character $\varphi:I_u\rightarrow \mathbb{F}^\times_{\ell^{h}}$ factors through $I_u^t$ and $\varphi: I_u^t\rightarrow \mathbb{F}^\times_{\ell^{h}}$ can be written as $\varphi=\psi_1^{t_1}\cdots \psi_h^{t_h}$, where $\psi_i$ are fundamental characters of level $h$ and $0\leq t_i\leq \ell-1$. The numbers  $t_i/e$ are called the \textbf{tame Inertia Weights} of the character $\varphi$ at $u$, where  $e=e(K_u/\mathbb{Q}_\ell)$ be the ramification index at $u$. For any finite extension $L_w/K_u$ such that $w|u$, if $e(L_{w}/K_u)< (\ell-1)/\mathrm{max}\{t_j~|~1\leq j\leq h\}$, the tame inertia weights of $\varphi|_{I_u^t}$ and $\varphi|_{I_w^t}$ are equal. Note that the residual representation $\bar{\rho}^{ss}|_{G_u}$ is tamely ramified and $\bar{\rho}^{ss}|_{I_u^t}$ is direct sum of characters $\varphi_i:I_u^t\rightarrow \mathbb{F}^\times_{\ell^{h_i}}$. The \textbf{tame inertia weights of $\rho$ at $u$} is defined to be  the tame inertia weights of the corresponding residual representation $\bar{\rho}^{ss}$ which is  the union of the sets of tame inertia weights of $\varphi_i$ for all $i$. In other words, for any $\mathbb{F}_\ell$-representation $W$ of $I_u$, the tame inertia weights of $W$ are the tame inertia weights of all the Jordan-H{\"o}lder quotients of $W$. 
		The fundamental character of level $1$ for $G_{u}$, $\psi_1: I_{u}^t\rightarrow \mathbb{F}_\ell^*$, is defined as follows: fix uniformizers of $\mathbb{Z_\ell}$ and $\mathcal{O}_u$, say, $\pi_\ell$ and $\pi_u$  respectively, we have $\pi_u^{e}=\pi_\ell$. Then $\psi_1(\sigma):=\dfrac{\sigma(\pi_u)}{\pi_u}$. We have $\bar{\chi}_\ell|_{I^t_u}=\psi_1^{e}$, where $\chi_\ell:G_{\mathbb{Q}_\ell}\rightarrow \mathbb{Z}_\ell^*$ is the cyclotomic character. The tame inertia weight of $\chi_\ell|_{G_{K_u}}$ is $[e]/e$, where for an integer $c$, $[c]$ denotes a unique integer in $\{0,1,\cdots \ell-1\}$ which is congruent to $c$ modulo $\ell-1$.  Hence the tame inertia weights of  $\oplus_{i=1}^n \chi_\ell^{a_i}$ are $\{[ea_1]/e,\cdots, [ea_n]/e\}$.  Tame inertia weights are invariant under any finite unramified extension of $K_u$.

		We recall a result due to Caruso which relates the tame inertia weights and the Hodge Tate weights of any $\ell$-adic representation of $G_K$ at places $u\mid \ell$ for $\ell$ sufficiently large. 
		
		\begin{prop}\cite{Ca06}\label{thm:caruso}
			Let $b$ be a positive integer and $\rho:G_K\rightarrow GL_n(\Q_{\ell})$ be an $\ell$-adic Galois representation  with Hodge-Tate weights at $u$ are in $ [0,b]$. If $e({K_{u}/\Q_{\ell}})b<\ell-1$, then the tame inertia weights of  $\bar{\rho}|_{G_{u}}$ are in  $ [0,b]$.
			
		\end{prop}
		
		\section{Generalizations of Faltings' Finiteness Theorems}

		We recall few results from \cite{DR23} which generalizes Faltings' finiteness theorems for pure semi-simple $\ell$-adic representations to the context of potential equivalence. 
		
		Let $F$ be a field and $G$ be a group. Two representations of a group are said to be {\em potentially equivalent} if they become isomorphic when we restrict to a subgroup of finite index.  Given a representation $\rho: G\to GL_n(F)$, the {\em $m$-power character} $\chi^{[m]}_{\rho}$ is defined to be the function $g\mapsto \mbox{Tr}(\rho(g^m))$ for any $g\in \Gamma$. Two representations  $\rho_1, \rho_2: G \rightarrow GL_n(F)$ are said to be {\em $m$-character equivalent} if their $m$-power character functions coincide. When $F$ is a nonarchimedean local field of characteristic zero we have the following equivalent criteria for potential equivalence of two representations.

		\begin{lem}\label{lem:pot-meq}\cite[Theorem 4]{DR23}
			Let $K$ be a number field and $F$ a nonarchimedean local field of characteristic zero. Suppose  $\rho_1$ and $\rho_2$ are continuous, semisimple representations of $G_K$ to $GL_n(F)$. Then the following are equivalent: 
			\begin{enumerate}
				\item The representations $\rho_1$ and $\rho_2$ are potentially equivalent. 
				
				\item There exists a positive integer $m$ depending only on $n$ and $F$ such that at a set of places $T$ of upper density one of $K$ such that order of $\rho_i(I_v)$ divides $m$ for $i=1,2$ and we have
				\begin{equation}\label{eqn:assertion3}
					\mbox{Tr}~\rho_1(F_v ^m)=\mbox{Tr}~\rho_2(F_v^m ),\quad v\in T.
				\end{equation}

			\end{enumerate}
			
		\end{lem}

		Further, we have a finiteness criteria to check potential equivalence of two $\ell$-adic representations extending the finiteness criteria of Faltings. 
		
		\begin{lem}\cite[Theorem 7]{DR23}\label{thm:fin}
			Let $K$ be a number field. Let us fix a natural number $n$, a rational prime $\ell$  and a nonarchimedean local field $F$ of residue characteristic  $\ell$. Let  $S$ denote a finite set of places of $K$ containing the archimedean places and the places of $K$ dividing $\ell$ of $K$.	Then there exists a finite set $T$ of finite places of $K$ disjoint from $S$, with the following property: suppose  $\rho_1, \rho_2:G_K\longrightarrow GL_n(F)$ are two continuous  semisimple representations of $G_K$. If there exists a positive integer $m$ such that  for $v$ not in $S$, the inertia at $v$ acts by scalar matrices of order dividing $m$ and the representations satisfy the following hypothesis: 
			$$\mbox{Tr}~\rho_1(F_v^m )=\mbox{Tr}~\rho_2(F_v^m ), \quad \text{for}~v\in T.$$
			Then $\rho_1$ and $\rho_2$ are potentially isomorphic. 
			
		\end{lem}

		\begin{lem}\cite[Corollary 7]{DR1}\label{DR1_cor7} Let $v$ be a finite place of $K$ at which the abelian variety $A_v$ has totally bad reduction and acquires good reduction over a quadratic extension of $K_v$. Assume further that the residue characteristic of $K_v$ is not two. Then for any rational prime $\ell$ coprime to the residue characteristic of $K_v$, the inertia group $I_v$ acts via scalars $\{\pm 1\}$ on the Tate module $V_{\ell}(A_v)$. 
		\end{lem}
		\begin{proof}[Sketch of Proof]
			Let $v$ be a finite place of $K$ at which the abelian variety $A_v$ has totally bad reduction and has potential good reduction.
			Then by \cite[Proposition 6]{DR1}, for any rational prime $\ell$ coprime to the residue characteristic of $K_v$, the subspace $V_{\ell}(A_v)^{I_v}$ of inertial invariants of the Tate module is trivial. Suppose $L_w$ is a quadratic extension of $K_v$ at which $A_v$ acquires good reduction, and let $I_w$ denote its inertia subgroup. By Serre and Tate (\cite{ST68}), $A$ has good reduction at a finite place $v$ of $K$ if and only if the representation of $G_K$ on the Tate module $V_\ell(A)$ attached to $A$ is unramified at $v$ for any $\ell$ not equal to the residue characteristic at $v$. By hypothesis, the extension $L_w$ over $K_v$ is ramified, and the action of the index two subgroup $I_w$ of $I_v$ on $V_\ell(A)$ is trivial. Now since $V_{\ell}(A_v)^{I_v}=0$, i.e., there exists no invariant vector for $I_v$ action on $V_\ell(A)$, the action of  $I_v/I_w\cong \mathbb{Z}/{2\mathbb{Z}}$ on $V_\ell(A_v)$ can't fix any non zero vector, i.e., the nontrivial element of $I_v/I_w$ acts by $-1$ on $V_\ell(A_v)$. 
		\end{proof}

		\begin{lem}\label{le1}
			Let $g>0$ be an integer and $\ell_0$ be a prime. Let $\beta_{\ell_0}$ be the set of continuous, semisimple $\ell_0$-adic representations  $\rho: G_K\rightarrow GL_{2g}(\mathbb{Q}_{\ell_0})$ such that $\rho(I_v)$ is a scalar matrix for finite places $v\nmid \ell_0$ of $K$ and  further $\rho$ satisfies the following conditions:
			\\	
			there exists an abelian variety $A/K$ with the $K$-isomorphism class $[A]\in \mathcal{Z}(K,g)$ such that
			$$\rho|_{G_{K_A}}\cong \rho_{A,\ell_0}|_{G_{K_A}},~ \text{where}~K_A=K(A[12]).$$
			Then the set $\beta_{\ell_0}$ is finite, up to potential equivalence.

		\end{lem}
		\begin{proof}

			The abelian varieties $A/K$  with $[A]\in \mathcal{Z}(K,g)$ have potential good reduction everywhere. Hence by Raynaud's criteria (section \ref{Dg}), $A$ has good reduction over $K_A=K(A[12])$ at all the places of $K_A$ and there exists an integer $D_g$ such that $[K_A:K]$ divides $D_g$. Let $v_{K_A}$ be a finite place of $K_A$ lying above any finite place $v$ of $K$ not dividing $\ell_0$ and $f_{K_A}=[k_{v_{K_A}}:k_v] $, where $k_{v_{K_A}}$, $k_{v}$ are the residue fields at $v_{K_A}$ and $v$ respectively. Since $K_A$ is a Galois extension over $K$ and $[K_A:K]$ divides $D_g$, so $D_g/f_{K_A}$ is an integer.  We have $\textrm{det}(XI_{2g}-\rho_{A,\ell_0}(F_v^{D_g}))=\textrm{det}(XI_{2g}-\rho_{A,\ell_0}(F_{v_{K_A}}^{D_g/f_{K_A}}))$. 
			
			Since $\rho|_{G_{K_A}}\cong\rho_{A,\ell_0}|_{G_{K_A}} (=\rho_{A_{K_A},\ell_0})$,  we have $\rho|_{G_{K_A}}$ is unramified at $v_{K_A}$ and so we get 
			$\mathrm{ord}(\rho(I_v))~|~D_g$. Further, we have,
			
			\begin{equation}\label{eq:D_g}
				\textrm{det}(XI_{2g}-\rho(F_v^{D_g}))=\textrm{det}(XI_{2g}-\rho_{A,\ell_0}(F_v^{D_g})).
			\end{equation}
			By \cite[p. 499]{ST68}, the polynomial on  RHS of equation \eqref{eq:D_g} is in $\mathbb{Z}[X]$ and its roots are pure of Weil  weight ${D_g}/2$ and of bounded degree.  It can be seen that there are only finitely many Weil numbers attached to a finite field of bounded degree and weight. Consequently, there are only finitely many choices for the roots of the characteristic polynomial of $\rho(F_v^m)$ for $v\in T$ as in Lemma \ref{thm:fin}  which determines $\rho$ up to potential equivalence, where $m=D_g$.  Hence $\beta_{\ell_0}$ is finite, up to potential equivalence. 
			
		\end{proof}
		\subsection{Finiteness of $\bar{K}$-isomorphism classes of Abelian Varieties $A/K$:}
		In \cite{MW93}, Masser-W\"ustholz provided a quantitative version of the finiteness result of Faltings \cite{Fa83}, namely, if $A$ is an abelian variety defined over a number field $K$, there are only finitely many $K$-isomorphism classes of abelian varieties $A^*$ defined over $K$ that are $K$-isogenous to $A$. They obtained a new proof of this theorem using `transcendence' techniques. Bost, in \cite{Bos96}, discussed that new approach and its applications to achieve various results concerning abelian varieties over number fields. We are interested in the following finiteness result obtained as a corollary of the isogeny estimates obtained by  Masser-W\"ustholz. 
		
		\begin{prop}\cite[Corollary 3.2]{Bos96}\label{fin_isom_isog}
			For any abelian variety $A$ defined over a number field $K$, there are only a finite number of $\bar{K}$-isomorphism classes of abelian varieties $B$ defined over $K$ and $\bar{K}$-isogenous to $A$.
		\end{prop} 
		
		Using the above proposition finiteness result obtained in \cite{DR23} can be rewritten as:
		
		\begin{prop}\label{fin_pot__isom}
			Let $K$ be a number field and $S$ be a finite set of
			places of $K$ containing the infinite places. 	Then, up to $\bar{K}$-isomorphism, there are only finitely many abelian varieties over $K$ of dimension $g$, such that the inertia at a place $v$ not in $S$ acts by scalars $\{\pm 1\}$ on a Tate module  $V_\ell(A)$ for $\ell$ coprime to the residue characteristic at $v$. 
		\end{prop}
		
		\begin{proof}
			In \cite[Corollary 14]{DR23}, the authors obtained the finiteness of $\bar{K}$-isogeny classes of abelian varieties $A$ defined over $K$ of dimension $g$ satisfying the hypothesis. The proposition follows by using Proposition \ref{fin_isom_isog}.
		\end{proof}

		\section{Classification of certain $m$-compatible systems}
		Due to Ozeki, we have a structure theorem for certain compatible systems of $\ell$-adic representations (\cite[Theorem 2.8, Proposition 2.13]{Oz13}). In particular,
		for a  $\mathbb{Q}$-integral strictly compatible system  $(\rho_\ell)_\ell$ of $n$-dimensional geometric semisimple $\ell$-adic representations of $G_K$ if there exists an infinite set $\Lambda$ of primes such that for any $\ell\in \Lambda$,  $\rho_\ell$ is semistable at some place $u$ of $K$ dividing $\ell$ with bounded Hodge-Tate weights and semisimplication of the residual representation, $\bar{\rho_\ell}^{\mathrm{ss}}\cong \oplus_{i=1}^n\bar{\chi}_\ell^{a_{\ell,i}}$, where $a_{\ell,i}$ are integers, then there exists integers $b_i$ independent of $\ell$ such that $\rho_\ell\cong\oplus_{i=1}^n \chi_\ell^{b_i}$. They further generalized it when $\rho_\ell$ is potentially semistable at some place $u$ lying above $\ell$ with bounded inertial levels.  
		Motivated by these results we obtain the following structure theorem for certain $m$-compatible systems of $\ell$-adic representations: 
		
		\begin{lem}\label{Main_lemma} Let $\left(\rho_{\ell}\right)_{\ell}$ be an  $\mathbb{Q}$-integral strictly $m$-compatible system of $n$-dimensional continuous, pure semisimple $\ell$-adic representations $\rho_\ell:G_{K}\rightarrow GL_n(\mathbb{Q}_\ell)$ with $m$-ramification set $S$ whose Weil weights are  bounded by $w$.

			Suppose  there exists an infinite set of rational primes $\Lambda$ satisfying the following conditions:
			
			For any $\ell \in \Lambda$,  there exists a Galois extension $L^\ell$ of $K$ (depending on $\ell$) such that $[L^\ell:K]~|~m$  and  $\rho_\ell$ satisfies the following: 
			\begin{itemize}
				\item[1.]  $\rho_\ell|_{G_{L^\ell}}$ is unramified at all finite places outside the places of $L^\ell$ lying above places in $S_\ell$,
			
				\item[2.]\emph{(Hodge-Tate)} there exists integers $w_{1} \leq w_{2}$ and a place $u_\ell$ of $L^\ell$ lying above $\ell$ such that $\rho_\ell|_{G_{u_\ell}}$ is semistable at $u_\ell$  and the Hodge-Tate weights of $\rho_\ell$ at $u_\ell$ are in $\left[w_{1}, w_{2}\right]$ for any pair $\left(\ell, u_\ell\right)$,

				\item[3.]\emph{(Potentially Cyclotomic)} 
				\begin{equation}\label{res_equiv}
					\bar{\rho}_\ell^{\mathrm{ss}}|_{G_{L^\ell}}\cong \oplus_{i=1}^n \bar{\chi}_{\ell}^{a_{\ell,i}}|_{G_{L^\ell}},
				\end{equation}
				where $a_{\ell,i}$ are integers. 

			\end{itemize}
			Then, there exists an integer $c$ independent of $\ell$ such that 
			${\rho_\ell}$ and $\bigoplus_{i=1}^n {\chi}_{\ell}^{c}$ are potentially equivalent.
			In particular, the Weil weights of $\rho_\ell$ and $\bigoplus_{i=1}^n \chi_\ell^{c}$ are equal and independent of $\ell$. 
		\end{lem}

		\begin{proof}
			
			For any $\ell$, we can twist  $\rho_\ell$ by $\chi_\ell^r$ for a large positive integer $r$ and we may assume that $w_1\geq 0$.
			Let $\ell\in \Lambda$. Taking a finite place $u_\ell$ of $L^\ell$ as in condition $(2)$, we have the ramification index $e_{u_\ell}=e(L^\ell_{u_\ell}/{\mathbb{Q}_\ell})\leq m[K:\mathbb{Q}]$.  Put $e=\mathrm{lcm}_{\ell\in \Lambda}e_{u_\ell}$.
			We have
			$$e\leq (m[K:\mathbb{Q}])!,$$
			which is independent of $\ell$. Replace $\Lambda$ with its infinite subset such that $\ell\in \Lambda$ satisfies 
			$$\ell-1>w_2((m[K:\mathbb{Q}])!)>w_2 e> w_2e_{u_\ell}.$$
			Now using the theorem of Caruso (Proposition \ref{thm:caruso}) on upper bound of the tame inertia weights of $\bar{\rho}_\ell|_{G_{u_\ell}}$ (viewed as an $\mathbb{F}_\ell$-representation) and equation \eqref{res_equiv} implies that there exists an integer $b'_{\ell,i}\in [0, e_{u_\ell}w_2]$ which satisfies 
			\begin{equation}\label{eq_a}
				b'_{\ell,i}\equiv e_{u_\ell} a_{\ell,i}~\pmod{\ell-1}.
			\end{equation}
			Put	$b_{\ell,i}=b'_{\ell,i}e/{e_{u_\ell}}$. So for any $\ell\in \Lambda$, we have $b_{\ell,i}\in[0,ew_2]$ and 
			\begin{equation}\label{eq_b}
				b_{\ell,i}\equiv e a_{\ell,i}~\pmod{\ell-1}.
			\end{equation}
			\begin{claim}
				The set $\{b_{\ell,i}~|~1\leq i \leq n\}$ is independent of the choice of $\ell \in \Lambda$  large enough.
			\end{claim}
			
			\begin{proof}[Proof of Claim]

				Take a finite place $v_0$ of $K$ outside  $S_\ell $, let ${v_0}_{L^\ell}$ denote a place of $L^\ell$ dividing $v_0$. As $L^\ell/K$ is a Galois extension and $[L^\ell:K]|m$, we have $$\mbox{det}(XI_n-\rho_\ell(F_{v_0}^m))=\mbox{det}(XI_n-\rho_\ell(F_{{v_0}_{L^\ell}}^{{m}/f_{{v_0}_{L^\ell}}})),$$ where $f_{{v_0}_{L^\ell}}=[k_{{v_0}_{L^\ell}}:k_{v_0}]$.			Consider the polynomials $$P(X):=\mbox{det}(XI_n-\rho_\ell(F_{v_0}^{em}))\in \mathbb{Z}[X],~ \tilde{P}(X):=	\mbox{det}(XI_n-\oplus_{i=1}^n {\chi}_{\ell}^{b_{\ell,i}} (F_{v_0}^m))\in \mathbb{Z}[X].$$ 
				
		\noindent		As
				$\bar{\rho}_\ell^{\mathrm{ss}}|_{G_{L^\ell}} \cong \oplus_{i=1}^n \bar{\chi}_{\ell}^{a_{\ell,i}}|_{G_{L^\ell}}$, we have 
					\begin{align*}
					\mbox{det}(XI_n-\bar{\rho}_\ell(F_{v_0}^{em})=&\mbox{det}(XI_n-\bar{\rho}_\ell(F_{{v_0}_{L^\ell}}^{{em}/f_{{v_0}_{L^\ell}}}))\\=&
					\mbox{det}(XI_n-\oplus_{i=1}^n \bar{\chi}_{\ell}^{ea_{\ell,i}} (F_{{v_0}_{L^\ell}}^{m/f_{{v_0}_{L^\ell}}}))\\
					=&
					\mbox{det}(XI_n-\oplus_{i=1}^n \bar{\chi}_{\ell}^{b_{\ell,i}} (F_{{v_0}_{L^\ell}}^{m/f_{{v_0}_{L^\ell}}}))~[\text{by}~\eqref{eq_b}]\\
					=&\mbox{det}(XI_n-\oplus_{i=1}^n \bar{\chi}_{\ell}^{b_{\ell,i}} (F_{v_0}^m)) ~ .
				\end{align*}
				
				By assumption on Weil weights, the coefficients of $X^{n-i}$ in $P(X)$ and $\tilde{P}(X)$ have  absolute value bounded by $\binom{n}{i} q_{v_0}^{nmwe/2}$ and  $\binom{n}{i} q_{v_0}^{nmew_2}$ (as $b_{i,\ell}\in[0,ew_2]$) respectively.   
				Since $\Lambda$ is an infinite set, for $\ell$  sufficiently large ($\ell> 2 \mathop{max}_{0\leq i \leq n}\{\binom{n}{i} q_{v_0}^{nmwe/2},\binom{n}{i} q_{v_0}^{nmew_2}\}$, see Lemma \ref{OT}),  we have
				\begin{equation}\label{eq_claim_aa}
					\mbox{det}(XI_n-\rho_\ell(F_{v_0}^{em}))=
					\mbox{det}(XI_n-\oplus_{i=1}^n {\chi}_{\ell}^{b_{\ell,i}} (F_{v_0}^{m}))  \hspace{3pt} \in \mathbb{{Z}}[X].
 				\end{equation}
			 Now since $\rho_\ell$ is a member of a $m$-compatible system of representations (hence also $em$-compatible), so
			 the eigenvalues of LHS of the equation \eqref{eq_claim_aa} are independent of $\ell$, and hence the eigenvalues of $\oplus_{i=1}^n\chi_\ell^{b_{\ell,i}}(F_{v_0}^m)$  are independent of $\ell$, so $b_{\ell,i}$ are independent of $\ell$ for sufficiently large $\ell$. 
			\end{proof}
			As $\rho_\ell$ are pure, we have $b_{\ell,1}=\cdots=b_{\ell,n}$ and they independent of $\ell$ for $\ell\in\Lambda$ large enough, so put $b=b_{\ell,1}=\cdots=b_{\ell,n}$. By $em$-compatibility of the system $(\rho_\ell)_\ell$ from \eqref{eq_claim_aa} we have 
			\begin{equation}\label{eq_claim_a}
				\mbox{det}(XI_n-\rho_\ell(F_v^{em}))=
				\mbox{det}(XI_n-\oplus_{i=1}^n {\chi}_{\ell}^{b} (F_v^m))
			\end{equation}
				for any $\ell$ (not in the defect set $T$) and for any $v\notin S_\ell$. 
			
			\begin{claim}\label{Claim_a} For any prime $\ell$, we can define
				a continuous character 	${\chi}_{\ell}^{1/e}: G_{K'}\rightarrow \mathbb{Q}^{\times}_{\ell}$  which has values in $ \mathbb{Z}^*_{\ell}$ and ${({\chi}_{\ell}^{1/e})}^e=\chi_{\ell}$ for some finite Galois extension $K'/K$ such that $K\subset L^{\ell} \subset K'$. In particular, there is a positive integer $m'$ divisible by $m$, such that we have
				$$	\mbox{det}(XI_n-\rho_{\ell}(F_v^{m'}))=
				\mbox{det}(XI_n-\oplus_{i=1}^n ({\chi}_{\ell}^{1/e})^b (F_v^{m'})) $$	for any finite place $v\notin S_{\ell}$.
			\end{claim}
			
			\begin{proof}[Proof of Claim]
				Let $\mathfrak{m}_{\ell}$ be the maximal ideal of $\mathbb{Z}_{\ell}$. Fix an integer $k>1 ~(\geq 1/{(\ell-1)})$, and take a finite Galois extension extension $K'$ of $K$ with $K\subset L^\ell \subset K'$ such that $\chi_\ell(G_{K'})\subset 1+e\mathfrak{m}^k$. Then we take the composite of maps to get $\chi_{\ell}^{1/e}$:
				$$G_{K^{\prime}} \stackrel{\chi_{\ell}}{\rightarrow} 1+e \mathfrak{m}_{\ell}^k \stackrel{\log }{\rightarrow} e \mathfrak{m}_{\ell}^k \stackrel{1 / e}{\rightarrow} \mathfrak{m}_{\ell}^k \stackrel{\exp }{\rightarrow} 1+\mathfrak{m}_{\ell}^k \subset \mathbb{Q}_{\ell}^{\times} .$$
				Take a finite place $v\notin S_{\ell}$. Let $v_\ell$ be a place of $L^\ell$ lying above $v$. Since $K'$ and $L^\ell$ are Galois extensions of $K$, for any place $v'$ of $K'$ lying above $v_{\ell}$ of $L^{\ell}$, we have 
					\begin{align*}
				 \mbox{det}(XI_n-\oplus_{i=1}^n {\chi}_{\ell}^{b}(F_v^{d[L^{\ell}:K]}))
				 =&\mbox{det}(XI_n-\oplus_{i=1}^n {\chi}_{\ell}^{b} (F_{v_{\ell}}^{d[L^{\ell}:K]/f_{v_{\ell}}}))\\
				 =&\mbox{det}(XI_n-\oplus_{i=1}^n {\chi}_{\ell}^{b} (F_{v'}^{d[L^{\ell}:K]/f_{v'}}))\\
				 =&\mbox{det}(XI_n-\oplus_{i=1}^n ({\chi}_{\ell}^{1/e}(F_{v'}))^{bed[L^{\ell}:K] /f_{v'}})\\
				 =&	\mbox{det}(XI_n-\oplus_{i=1}^n ({\chi}_{\ell}^{1/e})^b (F_v^{ed[L^{\ell}:K]})),
				 \end{align*}
				  where $f_{v'}=[k_{v'}:k_v]$, $f_{v_{\ell}}=[k_{v_{\ell}}:k_v]$ (so $[k_{v'}:k_{v_{\ell}}]\cdot f_{v_{\ell}}=f_{v'}$), and
				  $d=[K':L^{\ell}]$. Since $[L^\ell:K]~|~m$, by the above equation and equation \eqref{eq_claim_a} we have
						$$	\mbox{det}(XI_n-\rho_\ell(F_v^{edm}))
				   = \mbox{det}(XI_n-\oplus_{i=1}^n {\chi}_{\ell}^{b}(F_v^{dm}))	
					= \mbox{det}(XI_n-\oplus_{i=1}^n ({\chi}_{\ell}^{1/e})^b (F_v^{edm})).$$
		\noindent		Taking $m'=edm$ we can write, 
					$$	\mbox{det}(XI_n-\rho_{\ell}(F_v^{m'}))=
				\mbox{det}(XI_n-\oplus_{i=1}^n ({\chi}_{\ell}^{1/e})^b (F_v^{m'})) .$$

			\end{proof}
			
			Hence using Claim \ref{Claim_a} and  Lemma \ref{lem:pot-meq}, ${\rho}_{\ell}$ and  $ \oplus_{i=1}^n ({\chi}_{\ell}^{1/e})^b$ are potentially equivalent  for any $\ell$. Since ${\rho}_{\ell}$ is Hodge-Tate at any place of $L^{\ell}$ lying above $\ell$, the number $b/e=: c$ is an integer. This proves the Lemma.


		\end{proof}

		\begin{lem}	\label{OT}\cite[Lemma 3.8]{OT14}
					Suppose $a$ be a non-zero integer and $C>0$ be a real number. If $a\equiv 0 \pmod{\ell}$ and $|\iota(a)|\leq C$ for any embedding $\iota: \bar{\mathbb{Q}}\hookrightarrow \mathbb{C}$, then $\ell<C$.     	
		\end{lem}

		\section{Proof of Main Theorem \ref{thm1} }
		
		First, we recall a result on the structure of the representation $\bar{\rho}_{A,\ell}:G_K\rightarrow \mathrm{GL}(A[\ell])\cong GL_{2g}(\mathbb{F}_\ell)$ for any abelian variety $A/K$ such that the $K$-isomorphism class $[A]\in \mathcal{A}(K,g,\ell)$.
		
		\begin{lem}\cite[Lemma 3]{RT08} \label{lem:RT}
			Let  $A/K$ be an abelian variety of dimension $g$. Then  $K(A[\ell])$ is an $\ell$-extension of $K(\mu_\ell)$ if and only if ${\bar{\rho}_{A,\ell}}^{ss}\cong \oplus_{i=1}^{2g} \bar{\chi}_{\ell}^{a_{i}}$, where $a_i$ are integers.	
		\end{lem}
		\begin{proof}[Proof of Theorem \ref{thm1} ]
			
			Let us fix a rational prime $\ell_0$.
			Suppose $$\mathscr{M}(K,g,\ell)\neq\emptyset$$
			for infinitely many rational primes $\ell$. Hence, $\bar{\mathscr{M}}(K,g,\ell)$ is also non-empty for infinitely many rational primes $\ell$.  For every such primes $\ell$,  consider the $\ell_0$-adic representation $\rho_{A,\ell_0}$ such that $\rho_{A,\ell_0}\in \beta_{\ell_0}$, where $A/K$ is an abelian varieties  with $\bar{K}$-isomorphism class $[A]\in \bar{\mathscr{M}}(K,g,\ell)$. So for such primes  $\ell$, every $[A]\in \bar{\mathscr{M}}(K,g,\ell)$ determines an element $\beta_{\ell_0}$, up to potential equivalence. As $\beta_{\ell_0}$ is finite, up to potential equivalence (Lemma \ref{le1}),  there exists an $\ell_0$-adic representation $\rho_{\ell_0}:G_K\rightarrow GL_{2g}(\mathbb{Q}_{\ell_0})$ in $ \beta_{\ell_0}$, up to potential equivalence, such that for infinitely many primes $\ell$, there exists an element  $[A]\in \bar{\mathscr{M}}(K,g,\ell)$ satisfying 
			$$\rho_{\ell_0}|_{G_{K_A}}\cong \rho_{A,\ell_0}|_{G_{K_A}},$$
			where  $K_A=K(A[12])$. The representation $\rho_{\ell_0}$ extends to a $\mathbb{Q}$-integral strict $D_g$-compatible system $(\rho_\ell)_\ell$ of $2g$-dimensional semisimple $\ell$-adic representations whose $D_g$-ramification set and defect set are empty. Let $\Lambda$ denote the infinite set of places as obtained above, i.e., 
			\begin{align*}
				\Lambda=\{\ell~|~&\rho_{\ell_0}|_{G_{K_A}} \cong \rho_{A,\ell_0}|_{G_{K_A}} \text{ for some abelian variety $A/K$ with}~ [A]\in \bar{\mathscr{M}}(K,g,\ell),\\
				& \text{and $K_A(A[\ell])$ is an $\ell$-extension of $K_A(\mu_\ell)$, where $K_A=K(A[12])$ }   \}.
			\end{align*}
			For any $\ell \in \Lambda$, pick an abelian variety $A$ such that $[A]\in \bar{\mathscr{M}}(K,g,\ell)$. Consider the representation $\rho_\ell:=\rho_{A,\ell}$. Since $\rho_{\ell_0}|_{G_{K_A}} \cong \rho_{A,\ell_0}|_{G_{K_A}}$, and $K_A/K$ is a finite Galois extension with the degree  $[K_A:K]$  dividing $D_g$, we have
			\begin{align*}
				\textrm{det}(XI_{2g}&-\rho_{\ell_0}(F_v^{D_g}))=\textrm{det}(XI_{2g}-\rho_{A,\ell_0}(F_v^{D_g}))=\textrm{det}(XI_{2g}-\rho_{A,\ell_0}(F_{v_{K_A}}^{D_g/{f_{K_A}}}))\\	
				&=\textrm{det}(XI_{2g}-\rho_{A,\ell}(F_{v_{K_A}}^{D_g/{f_{K_A}}}))=\textrm{det}(XI_{2g}-\rho_{A,\ell}(F_v^{D_g}))=\textrm{det}(XI_{2g}-\rho_{\ell}(F_v^{D_g})),
			\end{align*}
			for	$v\notin\{\text{places of $K$ lying above}~\ell_0,\ell\}$.
			The third equality in the above equation is due to the fact that the Galois representations  $(\rho_{A_{K_A},\ell})_{\ell}$, corresponding to the base change $A_{K_A}$, gives compatible systems of $\ell$-representations with ramification set $S=\emptyset$ since $A$ has everywhere good reduction over $K_A$.

			Now choose a prime $\tilde{\ell}\in \Lambda$ and fix an abelian variety $B/K$ such that $[B]\in \bar{\mathscr{M}}(K,g,\tilde{\ell})$. So $\rho_{\ell_0}|_{G_{K_B}}\cong\rho_{B,\ell_0}|_{G_{K_B}}$, where $K_B=K(B[12])$. Then for any rational prime $\ell'\notin \Lambda$,  put $\rho_{\ell'}:=\rho_{B,\ell'}$. Now using a similar method used as above we can show that
			$$	\textrm{det}(XI_{2g}-\rho_{\ell_0}(F_v^{D_g}))=\textrm{det}(XI_{2g}-\rho_{B,\ell_0}(F_v^{D_g}))= 	\textrm{det}(XI_{2g}-\rho_{\ell'}(F_v^{D_g}))$$
			for	$v\notin\{\text{places of $K$ lying above}~\ell_0,\ell'\}$.
			
				By construction, the $D_g$-compatible system $(\rho_\ell)_\ell$ satisfies the following conditions for the above infinite set $\Lambda$.
			\begin{itemize}
				\item  $\rho_\ell|_{G_{K_A}}$ is unramified at all finite places outside the places of $K_A$ lying above $\ell$.
				
				\item $\rho_\ell|_{G_{K_A}}$ is semistable at any place of $K_A$ lying above $\ell$ (by Raynaud's criteria) and the Hodge-Tate weights are in $[0,1]$.
				
				\item by Lemma \ref{lem:RT}, $\bar{\rho}_\ell|_{G_{K_A}}\cong \bigoplus_{i=1}^{2g}\bar{\chi}_\ell^{a_{\ell,i}}|_{G_{K_A}}$.
			\end{itemize}  
			So $(\rho_\ell)_\ell$ satisfies the hypothesis of the Lemma \ref{Main_lemma} with infinite set $\Lambda$ as defined above, so there exists an integer $c$ independent of $\ell$ such that $\rho_\ell$ is potentially equivalent to  $\bigoplus_{i=1}^{2g}\chi_\ell^{c}$ for any $\ell$. Since $\rho_\ell$ is pure of Weil weight $1$ (by \cite{ST68}) and $\chi_\ell$ is pure of Weil weight $2$, we have  $1=2c$, which is a contradiction since $c$ is an integer. This proves that for sufficiently large $\ell$,  we have $$\mathscr{M}(K,g,\ell)=\emptyset.$$
			
		\end{proof}
		Given $K, g, \ell$, by Proposition \ref{DR1_cor7} and \ref{fin_pot__isom}, we have that the set of $\bar{K}$-isomorphism classes	$\bar{\mathscr{M}}(K,g,\ell)$ is finite. Hence using Theorem \ref{thm1}, the union 
	$$\bar{\mathscr{M}}(K,g):=\bigcup_\ell \bar{\mathscr{M}}(K,g,\ell)$$ is finite, which proves Corollary \ref{Main:cor1}.

		\begin{proof}[Proof of Corollary \ref{Main:cor2}]
			Let $A/K$ be such that the $K$-isomorphism class $[A]\in \mathcal{A}(K,g,\ell)$. Hence
			$K(A[\ell])$ is an $\ell$-extension of $K(\mu_\ell)$. Let $L=K(A[\ell])\cap K_A(\mu_\ell)$, where $K_A=K(A[12])$. Then $K(A[\ell])/L$  is an $\ell$-extension as $K(\mu_\ell)\subseteq L\subseteq K(A[\ell])$. Hence the compositum $K_A(A[\ell])= K(A[\ell]) K_A(\mu_\ell)$ is an $\ell$-extension of $K_A(\mu_\ell)$. So we have
			$$	\mathcal{A}^{\text{pot}}(K,g,\ell)\subset \mathscr{M}(K,g,\ell),$$
			and the Corollary follows from Theorem \ref{thm1} and Faltings' finiteness theorem.
			
		\end{proof}

			\subsection*{Acknowledgements} 
		     The authors thank Akio Tamagawa for pointing out major errors occurred in earlier versions of the article. The authors thank Yuichiro Taguchi,  Yoshiyasu Ozeki for their valuable suggestions which simplify the exposition. The authors further thank Anupam Saikia, Najmuddin Fakhruddin, Dipendra Prasad, C. S. Rajan for valuable discussions. Finally, the authors would like to thank KSCSTE-Kerala School of Mathematics for their support where the work has been done.

	\end{document}